\documentclass[12pt]{amsart}

\RequirePackage[nonfrench]{kotex}
\usepackage{kotex}

\usepackage{epsfig}
\usepackage{amsfonts}
\usepackage{amssymb}
\usepackage{amsmath}
\usepackage{amsthm}
\usepackage{amscd} 
\usepackage{amstext}
\usepackage{latexsym}
\usepackage{array}
\usepackage{graphicx, subcaption}
\usepackage{todonotes}

\usepackage[onehalfspacing]{setspace}

\numberwithin{equation}{section}

\usepackage{enumerate}

 \usepackage[backref,bookmarks=false]{hyperref} 

\newcommand{\ZZ}{\mathbf{Z}}
\newcommand{\CC}{\mathbf{C}}
\newcommand{\QQ}{\mathbf{Q}}

\newcommand{\OO}{\mathcal{O}}

\newcommand{\iddb}{i \partial \overline{\partial}}

\newcommand{\db}{\overline{\partial}}
\newcommand{\al}{\alpha}

\newcommand{\ve}{\vert  }

\newcommand{\ep}{\epsilon}

\newcommand{\ld}{\lambda}

\newcommand{\qa}{\quad}
\newcommand{\vp}{\varphi}

\DeclareMathOperator{\dv}{div}

\newcommand{\ut}{\tilde{u}}

\newcommand{\yr}{ Y_{\reg}}

\newcommand{\noi}{\noindent}

\newcommand{\dt}{\Delta}

\providecommand{\abs}[1]{\left|#1\right|}

\providecommand{\wt}[1]{\widetilde{#1}}

\theoremstyle{plain}

\newtheorem{theorem}{Theorem}[section]

\newtheorem{cor}[theorem]{Corollary}

\newtheorem{prop}[theorem]{Proposition}

\newtheorem{thm}[theorem]{Theorem}

\newtheorem{lemma}[theorem]{Lemma}

\newtheorem{proposition}[theorem]{Proposition}
\newtheorem{definition}[theorem]{Definition}

 \newtheorem{example}[theorem]{\textnormal{\textbf{Example}}}

\theoremstyle{remark}
\newtheorem{remark1}[theorem]{Remark}

\DeclareMathOperator{\red}{red}
\DeclareMathOperator{\Supp}{Supp}

\DeclareMathOperator{\reg}{reg}

\DeclareMathOperator{\sing}{sing}

\DeclareMathOperator{\mult}{mult}

\DeclareMathOperator{\ord}{ord}
\DeclareMathOperator{\Diff}{Diff}

\DeclareMathOperator{\Adj}{Adj}
\DeclareMathOperator{\adj}{Adj}

\DeclareMathOperator{\AAdj}{AAdj}
\DeclareMathOperator{\aadj}{AAdj}

\DeclareMathOperator{\Ex}{Ex}
\DeclareMathOperator{\exc}{Exc}
\DeclareMathOperator{\bigO}{O}

\begin{document}

\title{On $L^2$ extension from singular hypersurfaces}

\keywords{$L^2$ extension theorems; Singularities of pairs; Inversion of adjunction; Ohsawa measure; Kawamata log terminal; purely log terminal; plurisubharmonic}

\subjclass[2010]{}

\author{Dano Kim and Hoseob Seo }

\date{}

\maketitle

\begin{abstract}

\noi  In $L^2$ extension theorems from a singular hypersurface in a complex manifold, important roles are played by  certain measures  such as the Ohsawa measure which determine when a given function can be extended.  We show that the singularity of the Ohsawa measure can be identified in terms of singularity of pairs from algebraic geometry. Using this, we give an analytic proof of the inversion of adjunction  in this setting. Then these considerations enable us to compare various positive and negative results on $L^2$ extension from singular hypersurfaces. In particular, we generalize a recent negative result of Guan and Li which places limitations on strengthening such $L^2$ extension by employing a less singular measure in the place of the Ohsawa measure.

\end{abstract}

\section{Introduction}

 Let $X$ be a complex manifold and $Y \subset X$ a smooth hypersurface.  In their seminal paper~\cite{OT}, Ohsawa and Takegoshi showed the following $L^2$ extension from $Y$ to $X$ when $X \subset \CC^n$ is a bounded pseudoconvex domain and $Y$ is a hyperplane.  Let $\vp$ be a psh (i.e. plurisubharmonic) function on $X$. If a holomorphic function $u$ on $Y$ satisfies the finiteness of the RHS of the following \eqref{est0}, then there exists a holomorphic function $\ut$ on $X$ such that $\ut|_Y = u$ and satisfying \eqref{est0}: 
 
 \begin{equation}\label{est0}
  \int_X \abs{\ut}^2  e^{-\vp} dV_X   \le C  \int_Y \abs{u}^2 e^{-\vp}|_Y dV_Y 
\end{equation}

\noi where $dV_X$ and $dV_Y$ are smooth volume forms given by Lebesgue measures and  $C$ is a constant depending only on the diameter of $X$. Since \cite{OT}, there have been extensive further developments and important applications of $L^2$ extension theorems in both analysis and geometry (among so many others, cf. \cite{M93}, \cite{O4}, \cite{B96},  \cite{S98},  \cite{D00},  \cite{O5},  \cite{MV},  \cite{K10}, \cite{C11}, \cite{GZ12}, \cite{B13},  \cite{GZ15}, \cite{D15}, \cite{O7}, \cite{AM15}, \cite{CP20} for relevance to the subject of this article).  From the viewpoint of algebraic geometry, $L^2$ extension can be regarded as transcendental strengthening of cohomology vanishing theorems with far-reaching applications up to where vanishing theorems are not known to be applicable,  as in Siu's invariance of plurigenera \cite{S02}.

 It is then a fundamental question to ask whether such $L^2$ extension continues to hold when $Y$ is a  singular hypersurface. (By hypersurface, we mean a reduced subvariety of codimension $1$.)   Indeed, there have been positive results in this direction due to \cite{B96}, \cite{D00}, \cite{K10}, \cite{GZ15}, \cite{D15} and others.

 However, there also emerged recently a negative result  due to \cite{GL18} which states that $L^2$ extension `does not hold when $Y$ is singular' (cf. \cite[Thm. 3.31]{O18} for a related survey). This apparent dilemma raises the following natural question. 

\;

\noi \textbf{Question A.} 
 \emph{ How can we compare these positive and negative results on $L^2$ extension from singular hypersurfaces ?  }

\;

 We will answer this question in this paper by providing a uniform viewpoint in comparing these results.    The crux of the comparison lies in the volume form $dV_Y$ used in the $L^2$ norm on $Y$ (the RHS  of \eqref{est0}). \footnote{Although not necessary for this paper, we remark that another aspect of comparison could be on the curvature conditions in $L^2$ extension theorems, especially for compact $X$.}  When $Y$ is singular,  $dV_Y$ needs to become a \emph{singular} volume form, i.e. a \emph{measure} due to  technical but natural reasons arising from the used methods of $L^2$ estimates for $\db$.    On the other hand, the negative result of \cite{GL18} can be understood as saying that one cannot keep  using a smooth volume form (i.e. restriction to $Y$ of a smooth form on $X$) $dV_Y$ when $Y$ is singular : see Theorem~\ref{glgl}.  
 
 As the measure $dV_Y$ acquires singularities, it becomes significantly more nontrivial to satisfy the finiteness of the RHS of \eqref{est0}.  A priori, less singular $dV_Y$ will give a stronger $L^2$ extension theorem.  We remark here that the technical implementations of the {$L^2$ methods} used in the above positive results have quite some variations among them, hence their individual  $dV_Y$'s emerged in different and independent shapes from each other. 
But we will see that the measures turn out to give equivalent $L^2$ conditions.  
   More precisely,  we first recall the above positive results  \footnote{For the purpose of this paper, we recall here only the case of $Y$ being irreducible of codimension $1$. For the full statements, see the original papers. }   together with their measures  ${dV_Y}$'s : 
 
 \begin{itemize}
 \item
  In the $L^2$ extension theorem  \cite[Thm 2.1]{B96} and \cite[Thm. 4.1]{D00},  the measure ${dV_Y}^1$ was given as  $\displaystyle \frac{1}{\abs{df}^2} dV$ where $f=0$ is an equation defining $Y$ and  $dV$ is a smooth volume form on $Y$.

\item
 In the $L^2$ extension theorem  \cite[Main Thm.]{K10},  the measure ${dV_Y}^2$ was given in terms of `Kawamata metric'. For $Y$ of codimension $1$,  it can be equivalently given in terms of a divisor called the `different'  (see \eqref{ad} in \S 2 and Remark~\ref{kawadiff}).

 \item
 In the $L^2$ extension theorem  \cite[Thm. 2.8]{D15} and \cite[Thm. 2.1]{GZ15},  the measure  ${dV_Y}^3$ was defined as the Ohsawa measure  $dV[\psi]$ (see Definition~\ref{iber}) which generalizes a construction from \cite{O5}. 
 
 \end{itemize}

  As mentioned in \cite{D15}, ${dV_Y}^3$ and ${dV_Y}^1$  give the same $L^2$ conditions since their quotient is locally bounded.
    We will show that ${dV_Y}^3$ and ${dV_Y}^2$ also give the same $L^2$ conditions : this is the content of Theorem~\ref{dvpsi3} which we will discuss in the next subsection. This will complete one part of our answer to Question A.  For the other part concerning negative results, we will come back to Question A in \S 1.3 using the same viewpoint of singular volume forms $dV_Y$.

 As a representative of these positive results, we recall the statement of \cite[Thm. 2.8]{D15} in the special case (for all we need in this paper) of $Y$ of codimension $1$ and $X$ Stein: see  Theorem~\ref{extn} for more details.

\begin{thm}[=Theorem~\ref{extn}, J.-P. Demailly]\cite[Thm. 2.8]{D15}\label{extn0}
 Let $X$ be a Stein manifold and let $Y \subset X$ be an irreducible hypersurface. Let $s \in H^0 (X, \OO(Y))$ be a defining holomorphic section of $Y$.  Let  $\psi = \log \abs{s}^2$ be a quasi-psh function where the length $\abs{s}$ is taken with respect to a smooth hermitian metric of $\OO(Y)$. Under the curvature conditions \eqref{curv},  if a holomorphic function $u$ on $Y$ satisfies the finiteness of the RHS of the following \eqref{est1}, then there exists a holomorphic function $\ut$ on $X$ such that $\ut|_Y = u$ and satisfying \eqref{est1}:

\begin{equation}\label{est1}
  \int_{X} \abs{\ut}^2  e^{-\vp} \gamma (\delta \psi) e^{-\psi} dV_X  \le  \frac{34}{\delta} \int_{\yr}  \abs{u}^2 e^{-\vp}|_Y dV [\psi] 
\end{equation}
\noi where $\gamma$ is as defined in \eqref{gam} so that $\gamma (\delta \psi) e^{-\psi}$ is comparable to the Poincar\'e weight $\frac{1}{\abs{s}^2 (-\log \abs{s})^2}$.
 \end{thm}

\subsection{Local integrability of the Ohsawa measure}
  In order to state Theorem~\ref{dvpsi3},  let again $X$ be a complex manifold and $Y$ an irreducible hypersurface. Let $\psi$ be a quasi-psh function with \emph{poles along the divisor} $Y$ (in the sense as in the beginning of \S 3). Let $dV[\psi]$ be the Ohsawa measure   with respect to a smooth volume form $dV_X$.
    We will say that a measure is locally integrable if it has locally finite mass. The following is our first main result. 

\begin{theorem}\label{dvpsi3}
 
  Let $B$ be a $\QQ$-divisor on $X$ such that the support of $B$ does not contain $Y$. Write 
  $B = B_+ - B_-$  with $B_+ , B_- \ge 0$  \footnote{It does not matter for our purpose even if $B_+$ and $B_-$ share components.}.  Let $\vp_B := \vp_{B_+} - \vp_{B_-}$ where $\vp_{B_+}$ and $\vp_{B_-}$ are quasi-psh functions with poles along the effective divisors $B_+$ and $B_-$ respectively. Then the measure $e^{-\vp_B} dV[\psi]$ is locally integrable if and only if the pair $(Y^{\nu}, \Diff B)$ is klt (where $Y^{\nu}$ is the normalization of $Y$). 

\end{theorem}

Here klt (Kawamata log terminal) is a fundamental notion of singularities in algebraic geometry and  the divisor $\Diff 0$ is the `different', a well-known correction term for the lack of usual adjunction  : see \S 2.  
 The special case when $B=0$  gives  the following 

\begin{cor}\label{dvpsi}

 In the above setting,  $dV[\psi]$ is locally integrable if and only if $Y$ is normal and the pair $(Y,0)$ has canonical singularities. 
\end{cor}

See also Remark~\ref{Li21}. Our understanding is that such a geometric characterization in Theorem~\ref{dvpsi3} could be rather surprising from the viewpoint of pure analysis.  Indeed, originally the concept of the Ohsawa measure arose within the contexts of  analysis~\cite{O4}, \cite{O5} motivated by the theory of interpolation and sampling (cf. \cite{Se93}, \cite{O20}). In this regard, Theorem~\ref{dvpsi3} can be seen as an instance of application of algebraic geometry to several complex variables. 
 
The idea of the proof of Theorem~\ref{dvpsi3} is as follows. The source of relation between  $\Diff$ and $dV[\psi]$ lies in that both of them are related to Poincaré residues when $Y$ is smooth. In general, the effect of resolution of singularities should be accounted. We use the important insight of Demailly~\cite{D15}  that the measure $dV[\psi]$ can be expressed as the direct image (along the resolution map) of a measure with poles along some divisor $\Gamma$ appearing in a log resolution (cf. \cite[p.5]{KM98}) of $(X, Y+B)$. We  combine this with the fact that the different $\Diff(B)$ is given as the divisorial pushforward of the same divisor $\Gamma$ along the log resolution.  Along the way, we revisit the fundamental properties of the Ohsawa measure in a detailed manner in \S 3.

\begin{remark1}

The hypersurface $Y$ is assumed  irreducible in Theorem~\ref{dvpsi3}, which is sufficient for the purpose of this paper. It is certainly possible to give more statements in the case of reducible $Y$, e.g. using $B$ in Theorem~\ref{dvpsi3}.  Here we do not pursue such generality in this direction since,  in a much more general setting of general codimension in \cite{K21},  such a  geometric characterization of local integrability of the Ohsawa measure will be given. \footnote{Some considerations based on \cite{D15}  in \S 3 of this paper are  used in \cite{K21}. }

\end{remark1}

\subsection{Inversion of adjunction}
 We now turn to more of algebraic geometry, where Theorem~\ref{dvpsi3} and Corollary~\ref{dvpsi} will lead to applications. In the minimal model program (MMP) and higher dimensional algebraic geometry (cf. \cite{KM98}), it is of fundamental importance to understand singularities of varieties and divisors. Let $X$ be a normal complex algebraic variety and $Y$ a prime divisor (i.e. an irreducible hypersurface taken as a divisor with coefficient $1$) on $X$. Let $B$ be a $\QQ$-divisor (not containing $Y$ in its support) on $X$ such that $K_X + Y +B$ is $\QQ$-Cartier. Consider the pair $(X, Y+ B)$ (without any further assumption such as log-canonical) where $Y$ appears with coefficient $1$. When $X$ and $Y$ are smooth, we have the usual adjunction formula $(K_X + Y + B)|_Y = K_Y + B|_Y$. In general, $Y$ is possibly non-normal and $B|_Y$ should be generalized to a divisor called the `different' $\Diff(B)$ which is defined on the normalization $Y^{\nu}$ (see \eqref{ad}). 

The following celebrated inversion of adjunction theorem  plays a crucial role in MMP by comparing singularities of varieties  in different dimensions (due to Shokurov~\cite{S92} in dimension 3 and to Koll\'ar~\cite[\S 17]{Ko92} in general).

\begin{theorem}[Inversion of adjunction]\label{inv} \cite{Ko92, Ko97}
 Assume that $B$ is effective. The pair $(Y^{\nu}, \Diff(B) )$ is klt if and only if the pair $(X, Y+ B)$ is plt in a neighborhood of $Y$. 
\end{theorem}
 
Here again, klt (Kawamata log terminal) and plt (purely log terminal) are fundamental notions of singularities in MMP: see \S 2.  The proof of Theorem~\ref{inv} uses an  argument involving relative vanishing theorems on a log resolution, called the connectedness lemma~\cite[Thm. 7.4]{Ko97}.

On the other hand, in the special case when $X$ and $Y$ are smooth,  a  simpler version  of Theorem~\ref{inv} follows immediately and directly from $L^2$ extension :  if $(Y, B|_Y)$ is klt, then $(X, B)$ is klt in a neighborhood of $Y$ \cite[Cor. 7.2.1]{Ko97} \footnote{Also see \cite[Thm. 2.5]{DK01} for a related more general statement.}.  It is natural to ask what happens to such a direct argument when $Y$ becomes singular.

 \;

\noi \textbf{Question B.} (J. Koll\'ar)
 \emph{Does the inversion of adjunction Theorem~\ref{inv} have an analytic proof using $L^2$ extension theorems that generalizes the direct argument of  \cite[Cor. 7.2.1]{Ko97} extending a function from $Y$ to $X$?} \footnote{In particular, here we do not mean  a proof which uses $L^2$ extension just \emph{somewhere} in the argument. } 

\;

 This question can be regarded as asking what has been lacking  in our understanding of the $L^2$ extension theorems in this regard.  Now using Theorem~\ref{dvpsi3}, we have the following second main result. 
 
\begin{thm}\label{main}

  The answer to Question B   is yes when $X$ is smooth. 

\end{thm}

The proof of Theorem~\ref{main} uses the $L^2$ extension Theorem~\ref{extn0} and a generalized version of the work of Guenancia~\cite{G12}  on analytic adjoint ideals. \footnote{In fact,  it is likely that similar ideas may lead to answering Question B in general. An $L^2$ extension theorem for this generality is already available from \cite{K10}, see Remark~\ref{kawadiff}.  }

While the above  algebraic method  is  efficient, the analytic arguments for Theorem~\ref{main} can be also used in the future, for example, to deal with the singularity of pairs involving plurisubharmonic functions  (as generalization of the classical pairs taken with divisors) : log resolutions may not exist for such pairs.

 \subsection{Some negative results on $L^2$ extension}
  Finally, going back to Question A, we revisit some negative results on $L^2$ extension theorems from an irreducible singular hypersurface. 
The main result of \cite{GL18} states that when $X$ is a (small enough) ball at $0 \in \CC^n$ and $0 \in Y \subset X$ an irreducible hypersurface, `$L^2$ extension holds' if and only if $0$ is a nonsingular point of $Y$. In fact, `$L^2$ extension' there in \cite{GL18} refers implicitly to a restrictive version in that, in the place of $dV_Y$ in \eqref{est0}, a smooth volume form $\wt{dV_Y}$ on $Y$ (i.e. the restriction of a smooth differential form of the right degree from $X$) is to be used. Note that such $\wt{dV_Y}$ is always locally integrable on $Y$ by the fundamental theorem of Lelong~\cite{L57} (cf. \cite[III (2.6)]{DX}).

 From the viewpoint of the Ohsawa measure $dV[\psi]$, the statements of  \cite[Thm. 1.2, Thm. 1.3]{GL18}  in the case of codimension $1$, can be alternatively reinterpreted  as follows. 
  
 \begin{theorem}[Q. Guan - Z. Li]\cite{GL18}\label{glgl}
  Let $Y \subset X$ be as in the $L^2$ extension Theorem~\ref{extn0}.

 \noi (1) If $Y$ is singular, then one cannot replace the Ohsawa measure $dV[\psi]$ in Theorem~\ref{extn0} with a smooth volume form on $Y$.

 \noi (2) If $Y$ is singular with a point $p \in Y$ such that $\mult_p Y \ge \dim X$, then one cannot replace the Ohsawa measure $dV[\psi]$ in  Theorem~\ref{extn0} with a locally integrable measure on $Y$. 
 
 \end{theorem}
 
 Let us explain these statements in more detail: suppose one considers a new statement of $L^2$ extension obtained from Theorem~\ref{extn0} by replacing each appearance of $dV[\psi]$ with a smooth volume form $dV_Y$ on $Y$. Such a new statement would be an improvement of Theorem~\ref{extn0} since the $L^2$ condition $\int_{\yr}  \abs{u}^2 e^{-\vp}|_Y dV_Y  < \infty$ would be definitely easier to satisfy than $\int_{\yr}  \abs{u}^2 e^{-\vp}|_Y dV [\psi] < \infty$. Theorem~\ref{glgl} (1) says that such a new statement is always false. Similarly for Theorem~\ref{glgl} (2).

 Note that under the condition of Theorem~\ref{glgl} (2), the pair $(Y^\nu, \Diff 0)$ is not klt by Proposition~\ref{notklt}, hence $dV[\psi]$ is not locally integrable due to Corollary~\ref{dvpsi}. 
  
 In order to describe the relation with Corollary~\ref{dvpsi}, we first  explain more about Theorem~\ref{glgl} (1) in the case when $Y$ is not smooth but the pair $(Y, 0)$ is klt. Then  $dV[\psi]$ is also  locally integrable on $Y$ by Corollary~\ref{dvpsi}. The proof of \cite[Thm. 1.2]{GL18} chooses a specific holomorphic function $u$ and a psh weight $e^{-\vp}$ such that  $\abs{u}^2 e^{-\vp}$ is locally integrable with respect to the smooth volume form $ dV_Y$ (but not locally integrable with respect to  $dV[\psi]$ since $dV[\psi]$ has some singularity even if it is locally integrable). Under the assumption that $L^2$ extension with the estimate \eqref{est0} holds for this $dV_Y$,  they derive contradiction from the finiteness of the LHS of \eqref{est0} using the Skoda division theorem.

 Using the inversion of adjunction, our third main result generalizes Theorem~\ref{glgl} (2) to every singular hypersurface that is not klt (in the sense that $(Y^\nu, \Diff 0)$ is not klt, cf. Propositions~\ref{notklt}, \ref{diff0}). 

\begin{thm}\label{neg3}
Let $Y \subset X$ be as in the $L^2$ extension Theorem~\ref{extn0}. If  $Y$ is singular such that $(Y^\nu, \Diff 0)$ is not klt,  then one cannot replace the Ohsawa measure $dV[\psi]$ in  Theorem~\ref{extn0} with a locally integrable measure on $Y$. 

 \end{thm}

 Indeed, there exist singular hypersurfaces (such as $x^2 + y^5 + z^7 = 0$ in $\CC^3$) whose cases are covered by Theorem~\ref{neg3} but not by Theorem~\ref{glgl} (2). 
 
 This paper is organized as follows. In Section 2, we introduce the notions of `singularity of pairs' from algebraic geometry. In Section 3, we define and study the Ohsawa measure. In Section 4, we revisit and  generalize the theory of analytic adjoint ideals due to Guenancia~\cite{G12} for our purpose. In Section 5, combining these ingredients, we complete the proofs of the main results.

\begin{remark1}\label{Li21}

 After this paper was posted on arxiv, the authors were informed by Zhenqian Li who kindly brought to their  attention his recent independent work \cite{Li21}, \cite{Li20} with results closely related to Theorem~\ref{dvpsi3} and Theorem~\ref{main}. He also pointed out slight strengthening  in Corollary~\ref{dvpsi} from its previous version. 
 
 \end{remark1}

 \noi \textbf{Acknowledgments.}  
  We would like to thank R. Berman, J. Kollár and T. Ohsawa for their helpful comments regarding the subject of this paper. We also thank Mario Chan for pointing out some inaccuracies in an earlier version regarding adjoint ideals (cf. \cite{G22}, \cite{C21}).      We would like to thank anonymous referees for useful comments. 
 This research was supported by Basic Science Research Programs through  National Research Foundation of Korea funded by the Ministry of Education: (2018R1D1A1B07049683) for D.K. and (2020R1A6A3A01099387) for H.S. Also H.S. was supported by the Institute for Basic Science (IBS-R032-D1-2022-a00).

\medskip

\section{Singularities of pairs}

  In this section, we will recall some fundamental notions of singularities of pairs from algebraic geometry (cf. \cite{KM98}, \cite{Ko97}). We will treat them in the usual algebraic setting (mainly for the purpose of the exposition on the inversion of adjunction)
  while
  it is well known that these notions also make sense for complex analytic spaces, cf.  \cite[\S 16.2]{GZ17},  \cite[(30)]{Ko21}, \cite[Def. 3.1]{F22}. Note that in our main results, we only use the case of (ambient) complex manifolds.

 Let $X$ be a normal complex variety. 
  Let $\Delta$ be a Weil $\QQ$-divisor on $X$ such that $K_X + \Delta$ is $\QQ$-Cartier.
 Let $f: X' \to X$ be  a proper birational morphism from a normal complex variety $X'$ and let $E$ be a prime divisor on $X'$.
  The \emph{discrepancy} $a(E, X, \Delta)$ of $E$ with respect to $(X, \Delta)$ is defined by the relation $K_{X'} = f^* (K_X + \Delta) + \sum_E a(E, X, \Delta)E$ (also see \cite[\S 3]{F22} for the generality of normal complex analytic spaces $X, X'$).

\begin{remark1}

  In the case of complex manifolds and normal complex analytic spaces, in general one does not have a global canonical divisor which corresponds to the canonical sheaf. However the notion of discrepancies (as defined below) is still globally well-defined, cf. \cite[\S 3]{F22}, \cite[(2.17)]{KM98}. 

\end{remark1}

  The pair $(X, \dt)$ is called \emph{klt} (i.e. Kawamata log terminal) if $a(E, X, \dt) > -1$ for every $E$ appearing in some $X'$ as above. The pair $(X, \dt)$ is called \emph{lc} (i.e. log canonical) if $a(E, X, \dt) \ge -1$ for every such $E$. Also the pair $(X, \dt)$ is called \emph{plt} (i.e. purely log terminal) if $a(E, X, \dt) > -1$ for every such $E \subset X'$ that is exceptional for $f: X' \to X$ with nonempty center on $X$. It is clear that we have the implication: klt $\to$ plt $\to$ lc. 
  
Now suppose that $\Delta = Y + B$, where $Y$ is a prime divisor (with coefficient $1$) and $\Supp B$ does not contain $Y$.  Let $\nu : Y^{\nu} \to Y$ be the normalization of $Y$.  Then a (unique) Weil $\QQ$-divisor $\Diff(B)$ on $Y^\nu$ satisfying  

\begin{equation}\label{ad}
(K_X + Y + B)|_{Y^\nu} \sim_{\QQ} K_{Y^\nu} + \Diff(B)
\end{equation}

\noi is defined and named as the `different' (\cite{S92}, \cite{Ko92}). Here the notation of restriction $|_{Y^\nu}$ means taking restriction to $Y$ first and then pullback by $Y^\nu \to Y$. We refer to \cite[\S 4.1]{Ko13} for a comprehensive treatment of the different including its definition and properties we use in this paper. \footnote{There was an earlier  definition of the different from \cite{S92}, \cite{Ko92} reducing to the dimension $2$ case by taking hyperplane sections. }

  We recall some basic properties of the different.    Let $f: X' \to X$ be a \emph{log resolution} of $(X, Y+ B)$, i.e. a proper birational morphism from a smooth variety $X'$ such that $\Ex(f)$ is a divisor and $\Ex(f) \cup f^{-1} (\Supp (Y+B))$  is a simple normal crossing divisor on $X'$.  Here the exceptional set $\Ex(f)$ is the set of points in $X$ where $f$ is not biregular. 
  Write $K_{X'} + \Delta_{X'}  = f^* (K_X + \Delta) = f^* (K_X + Y+ B)$. Let $Y'$ be the strict transform of $Y$ on $X'$ and $\tilde{f} : Y' \to Y^\nu$ the induced morphism. 
 
\begin{proposition}\label{prop1}

\begin{enumerate}
\item
  We have $K_{Y'} + (\Delta_{X'} - Y')|_{Y'}  \sim_{\QQ} \tilde{f}^*  (K_{Y^{\nu}} + \Diff (B) )$.
 
 \item  We have $ \Diff(B) = \tilde{f}_* ( (\Delta_{X'} - Y')|_{Y'} ) $ where $\tilde{f}_*$ denotes the divisorial pushforward along $\tilde{f}$. 
 
 \item
 
   If $(X, Y+B)$ is plt in a neighborhood of $Y$, then $(Y^\nu, \Diff (B))$ is klt. 
 \end{enumerate}
 \end{proposition}
 
 \begin{proof} 
 
 (1) follows from the above relation \eqref{ad} and the commutative diagram consisting of $Y' \subset X' \to X$ and $Y' \to Y^\nu \to Y \subset X$.  (2) follows from (1) (cf. \cite{H14}).    (3) is well known from \cite{Ko92}, \cite{Ko97}. 
 \end{proof} 
 \noi Here  the divisorial pushforward of a divisor $\sum_{i \in I} a_i D_i$  along a bimeromorphic morphism $g$ means that we take $\sum_{i \in I_0} a_i g(D_i)$ with $I_0 \subset I$ being the subset of the index $i$ for which the image $g(D_i)$ remains to be of codimension $1$ (cf. \cite[(2.40)]{Ko13}).

 If both $K_X +Y$ and $B$ are $\QQ$-Cartier, $\Diff(0)$ is also defined and we have $\Diff(B) = \Diff(0) + B|_{Y^{\nu}}$  \cite[(3.2.1)]{S92}. Note that $\Diff(0)$ can be indeed a nonzero divisor as in the following example. 

\begin{example}\cite[Ex. 4.3]{Ko13}
 Let $X$ be the normal surface defined by $(xy - z^m = 0)$ in $\CC^3$ (where $m \ge 1$). Let $Y$ be the Weil divisor defined by $(y=z=0)$. In this case, $\Diff(0)$ is equal to the $\QQ$-divisor  $(1 - \frac{1}{m}) [p]$ on $Y = Y^{\nu}$ where $[p]$ is the prime divisor on $Y$ given by the  point $p = (0,0,0) \in \CC^3$.   

\end{example}

 From this point on, we let $X$ be smooth. 

\begin{prop}\label{diff0}

 Let $X$ be a complex manifold and $Y$ a prime divisor on $X$. If   $Y$ is normal, then we have $\Diff (0) = 0$ as equality of Weil divisors on $Y$. 

\end{prop}

\begin{proof} 

 This follows immediately from \cite[Prop. 4.5 (1)]{Ko13} (see also \cite[Prop. 2.35 (3)]{Ko13} when $\dim X =2$) as follows.  Let $V \subset Y$ be a prime divisor. The coefficient of $\Diff (0)$ along $V$ is equal to zero if and only if $X$ and $Y$ are both regular at the generic point of $V$, which is the case since the dimension of the singular locus of $Y$ is strictly less than the dimension of $V$. 
\end{proof} 

\begin{remark1}

 By this proposition, when $Y$ is normal, we have $\Diff (B) = \Diff(0) + B|_Y = B|_Y$. Thus the statement of Theorem~\ref{inv} includes the inversion of adjunction as stated in \cite[Thm. 5.50 (1)]{KM98} and \cite[Thm. 7.5 (1)]{Ko97} (when $X$ is smooth). 
 
\end{remark1}

\begin{example}
 Let $X = \CC^2$ and $Y$ given by the equation $y^2 = x^2 (x+1)$. One blowup of the origin gives a log resolution of the pair $(X, Y)$ used in defining the different. In the above notation, $\tilde{f}: Y' \to Y^\nu$ is isomorphism and it is easy to check that $\Diff (0)$ is the sum of two points on $Y^\nu$ lying above the origin. 

\end{example}

 The following example shows that the inversion of adjunction as in Theorem~\ref{inv} does not hold when $B$ is not effective. 
 
\begin{example}[J. Kollár]

Let $X=\CC^2$, $Y=(x=0)$ and $B=2 (x-y=0)- 2 (x+y=0)$. Then $B$ restricts to the
zero divisor on $Y$, but the pair $(X, Y+B)$ is not even lc.

\end{example}

 \begin{proposition}\label{notklt}
 
  Let $X$ be a complex manifold and $Y$ a prime divisor on $X$. If there exists a point $p \in Y$ with $\mult_p Y \ge n$, then the pair $(Y^\nu, \Diff 0)$ is not klt.

 \end{proposition}
 
 \begin{proof} 
 
 From \cite[(3.9.4)]{Ko97}, if the pair $(X, Y)$ is plt, then $\mult_p Y < n$ for every point $p \in X$. Hence the given condition implies that $(X, Y)$ is not plt near the point $p$. By the inversion of adjunction Theorem~\ref{inv} when $B=0$, the pair $(Y^\nu, \Diff 0)$ is not klt.  
 \end{proof}

   \medskip
   
\section{Ohsawa measure}

 In this section, we will revisit the Ohsawa measure from \cite{O5}, \cite{D15} with a detailed treatment of this important notion.

We first have the following terminology.   Let $X$ be a complex manifold and let $D$ be an effective $\QQ$-divisor on $X$ such that $mD$ is Cartier for $m \ge 1$. 
   A quasi-psh function is said to be {\textbf{with poles along $D$}} if it is locally of the form $\frac{1}{m} \log \abs{u}^2 + \alpha$ where $u$ is a local equation of $mD$ and $\alpha$ smooth. Examples are given by  $\frac{1}{m} \log \abs{s}^2_h$ where 
  $s \in H^0 (X, \OO(mD))$ is a holomorphic defining section of $mD$ and $h$ is a smooth hermitian metric of the line bundle $\OO(mD)$. 
  
  Let $D$ be a not necessarily effective $\QQ$-divisor on $X$ such that $D = D_1 - D_2$ with $D_1, D_2 \ge 0$. A measure on $X$ is said to be {\textbf{with poles along}} $D$ if its local density functions with respect to the Lebesgue measure are of the form $e^{-\beta_1 + \beta_2}$ where $\beta_k$ is quasi-psh with poles along $D_k$.  Also a measure is said to be \textbf{with zeros along} $D$ if it has poles along $-D$.

\subsection{Ohsawa measure}
   
 Now let $Y \subset X$ be an irreducible reduced hypersurface.  Let again $D$ be an effective $\QQ$-divisor on $X$ whose support does not contain $Y$.  Let $\psi : X \to [-\infty, \infty)$ be a quasi-psh function with poles along the divisor $Y + D$.  
Let $dV_X$ be a continuous volume form on $X$ with nonnegative density with respect to the Lebesgue measure.  We only need the case when $dV_X$ has either (1) strictly positive density or (2) zeros along a divisor, thus we assume that $dV_X$ has zeros along a divisor $Z \ge 0$ (so that $Z=0$ corresponds to the case (1)). 
 
Let $T$ be the union of the singular locus of $Y$ and the support of $D$. Let $\yr$ be the regular locus of $Y$.  Following \cite{O5}, \cite{D15}, we have 

\begin{definition}\cite{D15} \label{iber}
 A positive measure $d\mu$ on $\yr$ is called the  \emph{ \textbf{Ohsawa measure}} of $\psi$ if it satisfies the following conditions:
 
 \begin{enumerate}[(a)]
 \item
   For every $g$, a real-valued compactly supported continuous function on $\yr \setminus T$ and every $\tilde{g}$, a compactly supported continuous extension of $g$ to $X \setminus T$, the relation 
 
 \begin{equation}\label{dvp}
\int_{\yr \setminus T} g \; d\mu = \lim_{t \to -\infty} \int_{\{x \in X \setminus T, t < \psi(x) < t+1 \}} \tilde{g} e^{-\psi} dV_{X} 
\end{equation}

\noi holds. 
\item The measure $d\mu$ on $\yr$ is the trivial extension of its restriction to $\yr \setminus T$. 

\end{enumerate}
\end{definition}

\begin{remark1}
 In \eqref{dvp}, $\lim$ is taken in the place of $\limsup$ as in \cite[(2.4)]{D15} (cf. \cite[p.4]{O5}). The fact that taking $\lim$ is enough in the setting of `neat' analytic singularities is due to \cite{D15} whose approach is followed in this section. In particular, the existence of the limit is given in Lemma~\ref{compu}.

\end{remark1}

If it exists, the Ohsawa measure is unique from basic properties of measures. The Ohsawa measure of $\psi$ will be denoted by $dV[\psi]$ where its dependence on the choice of $dV_X$ is suppressed. 
  The existence is due to \cite{D15} in this generality as a special case of \cite[Prop. 4.5]{D15} when $p=1$ and when $Y$ is of codimension $1$. For the sake of completeness and convenience for readers, we explain the arguments  adapted to our setting (cf. Remark~\ref{differ}) along with some addition and rearrangement of details. \footnote{In particular, independence of the choice of $\wt{g}$ is built into the definition of the Ohsawa measure and its existence is treated in an increasing order of generalities. }

\begin{prop}\label{iber2}
Let $X$, $dV_X$, $Y+D$ and $\psi$ be as in Definition~\ref{iber}: in particular, $\psi$ is quasi-psh on $X$ with poles along $Y+D$ and $dV_X$ has zeros along a divisor $Z \ge 0$.   The Ohsawa measure $dV[\psi]$ exists in the following order of generalities: 

\begin{enumerate}

\item When  $X$ is an open subset of $\CC^n$ with coordinates $(z_1, \ldots, z_n)$,  $Y$ is the hyperplane $(z_n = 0)$ and $Y+\red D + \red Z$ is an snc (i.e. simple normal crossing) divisor supported on the coordinate hyperplanes. (Here $\red D$ is the sum of the components of $D$ with coefficient $1$.)

\item  When $X$ is general,  $Y$ is an irreducible smooth hypersurface and $Y+ \red D + \red Z$ is an snc divisor.

\item When $X$ is general and $Y$ is an irreducible possibly singular hypersurface.

\end{enumerate}

  Also the measure $dV[\psi]$ has the property of  putting no mass on closed analytic subsets of $Y$. \footnote{See \cite[(1.3)]{BBEGZ} for more information on this property.} 
Moreover, in each of the above cases (1),(2),(3), we have the following properties (1'),(2'),(3') of $dV[\psi]$, respectively. 

\;

(1') If  $\psi = \log \abs{z_n}^2 + \alpha$ for a quasi-psh function $\alpha$ with poles along $D$ and $dV_X = e^\beta {d\ld_n}$ for a quasi-psh function $\beta$ with poles along $Z$, then we have 

$$ dV[\psi] = e^{-(\alpha - \beta)} {d\ld_{n-1}} $$

\noi where $d\ld_n$ is the Lebesgue measure on $\CC^n$. 

\;

(2')
$dV[\psi]$ is a measure on $Y$ with poles along the divisor $(D - Z)|_Y$. 

\;

(3') There exists a proper bimeromorhic morphism $Y' \to Y$ such that $dV[\psi]$ is the direct image  of a measure on $Y'$ with poles along a divisor as in \eqref{dnu}.

\end{prop}    

\begin{proof}

(1) When we restrict to $Y \setminus D \subset X \setminus D$, this is checked by computation which goes back to \cite[p.4]{O5}  (cf.  \cite{D15}).  For the sake of completeness, we provide details in this generality in  Lemma~\ref{compu}. Note that in (1'), we may assume that $\beta = 0$ and $\psi$ is of the form in \eqref{alpha}.    The measure given by (1') certainly has the property of putting no mass on closed analytic subsets.

 \;

(2)  Denote the LHS of \eqref{dvp} by $I(g)$. Since $e^{-\psi}$ can be locally non-integrable along (some subset of) the support of $D$, replace $\yr \setminus T$ by $Y \setminus D$ in \eqref{dvp}. Then $I(g)$ is a positive linear functional on the space of compactly supported continuous functons on $Y \setminus D$. (The required linearity of $I(g)$ reduces to (1) by using a partition of unity.)   Applying Riesz representation theorem, we get the corresponding measure on $Y \setminus D$ which we extend to a measure $d\mu$ on $Y$ trivially. Showing that this $d\mu$ satisfies \eqref{dvp} for a compactly supported $g$ on $Y \setminus D$ again reduces to the case (1) by using a partition of unity. The existence of the measure is shown. (Alternatively,  one can also show the existence by noting that $e^{-(\alpha - \beta)} {d\ld_{n-1}}$ in (1') is globally defined by Poincar\'e residue.)  The property (2') follows from (1').

\;

(3) The existence will be shown together with the property (3').  
Let $f: X' \to X$ be a log resolution of the pair $(X, Y+D)$.  Let $Y'$ be the strict transform of $Y$ on $X'$.    Let $Z$ be the effective divisor $K_{X'} - f^* K_X$ given by the contribution of the jacobian of $f$. Let $\yr' := f^{-1} (\yr)$. Consider the following defining condition of  the Ohsawa measure (say $d\nu$) of the quasi-psh function $f^* \psi$ with respect to the continuous volume form $f^* dV_X$ on $X'$.

  \begin{equation}\label{dvp1}
\int_{\yr'}  f^* g \; d\nu = \lim_{t \to -\infty} \int_{\{x \in X', t < f^* \psi(x) < t+1 \}} (f^* \tilde{g}) \; e^{-f^* \psi} f^* dV_{X}. 
\end{equation}

 Note that $f^* \psi$ has poles along $f^* (Y+D) = Y' + D'$ where $D'$ is defined by this relation and that $f^* dV_X$ has zeros along $Z$.      
 Hence we see that $d\nu$ is a measure on $Y'$ with  poles along the divisor

\begin{equation}\label{dnu}
 (D' - Z) |_{Y'} = (f^* Y + f^*D  - Y' - Z)|_{Y'} 
 \end{equation}
 
\noi  from the property (2'). 

Now, since $f: X' \to X$ is modification, we have the equality of the right hand sides of \eqref{dvp} and \eqref{dvp1} (before taking the lim) from change of variables.  Hence the direct image of $d\nu$ under $\yr' \to \yr$ satisfies the condition \eqref{dvp} and therefore gives the Ohsawa measure $dV[\psi]$. 
This concludes the proof. 
\end{proof}

\begin{remark1}

 We defined the Ohsawa measure $dV[\psi]$ in the generality of $\psi$ having poles along $Y+D$ even though in Theorem~\ref{extn0} and Corollary~\ref{dvpsi}, $\psi$ has poles along $Y$ only. The generality of $Y+D$  is more convenient in the above proof and also useful  for other applications. 

\end{remark1}

Now in order to state Lemma~\ref{compu}, let $Y+D$ be an snc divisor and let $U$ be the unit polydisk centered at the origin in $\CC^n$. Let $D = \sum^N_{i=1} a_i D_i$ where $N < n$ and $D_i$ is the prime divisor given by $(z_i = 0)$. 
 Let $\psi$ be a psh function on $U$ with poles along $Y + D$ given by
\begin{equation}\label{alpha}
\psi (z) = \log{ \left( \ve z_n \ve^2 \prod_{k=1}^{N}{\ve z_k \ve^{2 a_k}} \right) }.
\end{equation}

\begin{lemma}\label{compu}
For a continuous function $g$  with compact support on $Y \backslash D$, we have
\begin{equation} \label{WeightedOhsawa}
\lim_{t \to -\infty}{ \int_{ \{ x \in U : t < \psi(x) < t+1 \} }{ \widetilde{g} e^{-\psi} d\lambda_n  } }
= \pi \int_{Y \backslash D}{ \frac{ g }{ \prod_{k=1}^N \ve z_k \ve^{2a_k} } d \lambda_{n-1} }
\end{equation}
\end{lemma}
\begin{proof}
\;

\noi {\bf Step 1.} We will first show that the LHS of \eqref{WeightedOhsawa} does not depend on the choice of $\widetilde{g}$, compactly supported in $\CC^n \backslash D$. 
Denote two choices of $\widetilde{g}$ by $g_1$ and $g_2$. Then it suffices to show that

\begin{equation}\label{g1g2}
\lim_{t \to -\infty}{ \int_{ \{ x \in U : t < \psi(x) < t+1 \} }{ g_1 e^{-\psi} d\lambda_n  } }
\le \lim_{t \to -\infty}{ \int_{ \{ x \in U : t < \psi(x) < t+1 \} }{ g_2 e^{-\psi} d\lambda_n  } }.
\end{equation}

Since the supports of $g_1$ and $g_2$ are compact, they are uniformly continuous. Thus, for given $\epsilon>0$, one can find $\delta > 0$ such that $\ve g_i (x) - g_i(y) \ve < \epsilon$ whenever $\ve x - y \ve < 2\delta$, where $i =1,2$. Observe that for a $\delta$-neighborhood $N_\delta$ of $Y$ in $\CC^n$ and sufficiently negative $t < 0$, the set $\{x \in U : t < \psi(x) < t+1 \} \cap \left( \Supp g_1 \cup \Supp g_2 \right)$ is contained in $N_\delta$. For such $t$, we have (since $g_1 |_Y = g_2 |_Y$) 
\begin{equation}
\begin{aligned}
\int_{\{x \in U : t < \psi(x) < t+1\}}{ g_1 e^{-\psi} d\lambda_n }
&= \int_{\{x \in U : t < \psi(x) < t+1\}}{ g_1 \chi_{\CC^n \backslash A_{\epsilon_0}} e^{-\psi} d\lambda_n } \\
&\le \int_{ \{x \in U : t <\psi(x) < t+1 \}}{ (g_2 + 2\epsilon) \chi_{\CC^n \backslash A_{\epsilon_0}} e^{-\psi} d\lambda_n } \label{IndepOnExtension1}
\end{aligned}
\end{equation}
where $\epsilon_0$ is chosen so that the set $A_{\epsilon_0}$ defined by
$$
A_{\epsilon_0} : = \bigcup_{k=1}^N \{ (z_1,\ldots, z_n) \in \CC^n : \ve z_k \ve < \epsilon_0 \}
$$
does not intersect $\Supp g_1 \cup \Supp g_2$. Note that such $\epsilon_0$ can be found because $g_1$ and $g_2$ have compact supports in $\CC^n \backslash D$.
Hence the last line of (\ref{IndepOnExtension1}) is bounded above by

\begin{align*}
&\int_{\{x \in U : t < \psi(x) < t+1\}}{ g_2 e^{-\psi} d\lambda_n } + \frac{2\epsilon}{\epsilon_0^{2(a_1 + \cdots + a_N)}} \int_{ \{x\in U \backslash A_{\epsilon_0} : t<\psi(x)<t+1\}}{ \frac{1}{\ve z_n \ve^2} d\lambda_n } \\
&= \int_{\{x \in U : t < \psi(x) < t+1\}}{ g_2 e^{-\psi} d\lambda_n } + \frac{2C\epsilon}{\epsilon_0^{2(a_1 + \cdots + a_N)}}
\end{align*}
where $C$ is a constant independent of $t$. As $t$ tends to $-\infty$, we have
$$
\lim_{t \to -\infty}{ \int_{ \{ x \in U : t < \psi(x) < t+1 \} }{ g_1 e^{-\psi} d\lambda_n  } }
\le \lim_{t \to -\infty}{ \int_{ \{ x \in U : t < \psi(x) < t+1 \} }{ g_2 e^{-\psi} d\lambda_n  } }
+ \frac{2C\epsilon}{\epsilon_0^{2(a_1 + \cdots + a_N)}}.
$$
 Now letting $\epsilon$ converge to $0$, we have \eqref{g1g2}. \\ [10pt]

\noi {\bf Step 2.} In order to have \eqref{WeightedOhsawa}, thanks to Step 1, we may take $\widetilde{g}$ as
$$
\widetilde{g}(z_1, \ldots, z_n) = g(z_1,\ldots, z_{n-1}) (1 - \ve z_n \ve).
$$
Let $\epsilon$ be a positive number such that a polydisk $D_\epsilon^{n-1}(0) \subset \CC^{n-1}$ of radius $\epsilon$ does not intersect the support of $g$. Let $A := \prod_{k=1}^N \ve z_k \ve^{a_k}$. Then the integral on the LHS of (\ref{WeightedOhsawa}) is computed as follows:
\begin{align*}
\int_{ \{ x \in U : t < \psi(x) < t+1 \} }{ \widetilde{g} e^{-\psi} d\lambda_n  }
&= \int_{ \{ e^t < \ve z_n \ve^2 A^2 < e^{t+1} \} }{ \frac{g(z_1,\ldots, z_{n-1})}{A^2} \frac{1 - \ve z_n \ve}{ \ve z_n \ve^2 } d\lambda_n  } \\
&= 2\pi \int_{ \{ e^{t/2}A^{-1} < r_n < e^{(t+1)/2}A^{-1} \} }{ \frac{g(z_1,\ldots, z_{n-1})}{A^2} \frac{1-r_n}{r_n} dr_n d\lambda_{n-1} } \\
&= 2\pi \int_{ D^{n-1}_1 (0) \backslash D }{ \frac{g(z_1,\ldots, z_{n-1})}{A^2} \left[ \frac{1}{2} - e^{t/2}A^{-1} (\sqrt{e} - 1) \right] d\lambda_{n-1} }.
\end{align*}
Denote the last line by $I_1 - I_2$ where
\begin{align*}
I_1 &:= \pi \int_{D_1^{n-1}(0) \backslash D}{ \frac{g(z_1, \ldots, z_{n-1})}{\prod_{k=1}^N \ve z_k \ve^{2a_k} } d\lambda_{n-1} }, \\
I_2 &:= 2\pi (\sqrt{e}-1)e^{t/2} \int_{D_1^{n-1}(0) \backslash D}{\frac{g(z_1, \ldots, z_{n-1})}{A^3} d\lambda_{n-1}}.
\end{align*}
Since the support of $g$ does not meet $D_\epsilon^{n-1}(0)$, the integrand in $I_2$ is bounded. As $t$ tends to $-\infty$, $I_2$ converges to $0$, hence we obtain  (\ref{WeightedOhsawa}).
\end{proof}

\begin{remark1}

 In \cite{O5} (and later also used in \cite{GZ15}), the  Ohsawa measure was originally defined as the minimum element of a  certain partially ordered set of measures.

\end{remark1}

\begin{remark1}\label{differ}

Definition~\ref{iber} differs from its original counterpart in \cite{D15} where, supposing that the pair $(X, Y+ D)$ is lc, the Ohsawa measure $dV[\psi]$ is supported on the regular locus of $Y \cup D$ (where $D$ of course refers to its support), not just of $Y$. Definition~\ref{iber} is suitable for the purpose of this paper.
 \end{remark1}

\begin{remark1}

 In Definition~\ref{iber}, in fact it would be fine even if we replace $T$ by  another Zariski closed subset $T'$ containing $T$.  This is because when the property \eqref{dvp} is used in the proof of $L^2$ extension theorems (such as Theorem~\ref{extn}), we have the freedom to take off such a Zariski closed subset and perform the $L^2$ estimates in the complement when $X$ is Stein or quasi-projective, cf. \cite{S98}.

\end{remark1} 

\;

\medskip

\subsection{$L^2$ extension}  
  
   We will use the following  $L^2$ extension theorem stated with the Ohsawa measure.  It is a special case of \cite[Thm. 2.8]{D15} (and its singular metric version in Remark 2.9(b))  for $Y$ of codimension $1$ and $X$ Stein. (See \cite[Thm. 2.8]{D15}  for the original full statement.)

\begin{theorem}[J.-P. Demailly]\label{extn}

 Let $X$ be a Stein manifold.  Let $Y \subset X$ be an irreducible and reduced hypersurface.  Let $\psi : X \to [-\infty, \infty)$ be a quasi-psh function on $X$ with poles along $Y$.  Let $\omega$ be a K\"ahler metric on $X$.  
Let $F := K_X + L + M$ \footnote{using the additive notation for tensor products of line bundles} where $K_X$ is the canonical line bundle of $X$ and  $L$ and $M$ are line bundles on $X$. Let $h$ be a smooth hermitian metric of $L$. Assume that there is $\delta > 0 $ such that 

\begin{equation}\label{curv}
 i \Theta_{L, h} + \alpha \iddb \psi \ge 0
\end{equation}

\noi for every $\al \in [1, 1+\delta]$. Let $e^{-\vp}$ be a singular hermitian metric for $M$ with semipositive curvature current. For every holomorphic section $u \in H^0 (\yr, (K_X + L + M)|_{\yr})$ with

$$ \int_{\yr}  \abs{u}_{\omega,h,e^{-\vp}}^2 dV [\psi] < \infty,$$

\noi there exists a holomorphic section $\ut \in H^0 (X, K_X + L + M)$ whose restriction to $\yr$ is equal to $u$ such that 

\begin{equation}\label{eq1}
 \int_{X} \abs{\ut}^2_{\omega,h,e^{-\vp}}  \gamma (\delta \psi) e^{-\psi} dV_{X, \omega}  \le  \frac{34}{\delta} \int_{\yr}  \abs{u}_{\omega,h,e^{-\vp}}^2 dV [\psi] 
\end{equation}

\noi where $\gamma$ is defined by 

\begin{equation}\label{gam}
 \gamma(x) = e^{-\frac{1}{2} x} \text{ if } x \ge 0, \qa \gamma(x) = \frac{1}{1+x^2} \text{ if } x \le 0
\end{equation}

\noi as in \cite[Thm. 2.8]{D15}. 

 \end{theorem}

\begin{remark1}

Note that under the assumption of Theorem~\ref{extn}, the pair $(X, Y)$  is only `generically lc' while in  \cite[Thm. 2.8]{D15}, $(X, Y)$ is assumed to be lc.  \footnote{ In terms of \cite{D15}, `$\psi$ has log canonical singularities along $Y$' which is  the same as the pair being lc (i.e. log canonical) in the standard terminology.}  As is standard in the use of $L^2$ estimates in this context (cf. \cite{D82}, \cite{S98}), one can take a hypersurface $W \subset X$ such that $(X \setminus W, Y \setminus W)$ is lc and apply  \cite[Thm. 2.8]{D15} for this pair. The extended section $\ut$ on $X \setminus W$ further extends to $X$ by applying Riemann extension theorem (cf. \cite[Prop. 2.18]{K10}) due to the finite $L^2$ norm of $\ut$. Note that the non-lc singularity of $Y$ contributes to the singularity of $dV[\psi]$.

\end{remark1}

\section{Algebraic and analytic adjoint ideals}

 It is well known that being klt can be characterized as triviality of multiplier ideals (cf. \cite{Ko97}). For the notion of plt,  one can use the following adjoint ideals (cf. \cite[\S 9.3.E]{L}, \cite{EL97}, \cite{T10}):

\begin{definition}[Algebraic Adjoint Ideal]\label{Adj}

 Let $X$ be a complex manifold. Let $Y \subset X$ be an irreducible reduced hypersurface. Let $B$ be an effective $\QQ$-divisor on $X$ whose support does not contain $Y$.  Let $f: X' \to X$ be a log resolution of the pair $(X, Y+B)$. We define the algebraic adjoint ideal associated to this pair as 
 
 \begin{equation}\label{Adj1}
 \Adj (X, Y; B) := f_* \OO_{X'} (K_{X'} - \lfloor f^*(K_X + Y+ B) \rfloor  + Y')
 \end{equation}
 where $Y'$ is the strict transform of $Y$. 
 \end{definition}

  It is  known from \cite{T10} that this definition is independent of the choice of $f$ and that $\Adj (X, Y; B) = \OO_X$ if and only if $(X, Y+B)$ is plt in a neighborhood of $Y$.

On the other hand,  Guenancia~\cite{G12} defined the notion of  analytic adjoint ideals attached to a psh function  and showed that it coincides with the algebraic adjoint ideals when the psh function has analytic singularities. In this section, we point out its  generalization used in the proof of our main theorem. Namely, in \cite{G12}, the divisor $Y$ was assumed to be smooth but we will allow $Y$ singular. (See also \cite{G22}, \cite{C21}.)

\begin{definition}[Analytic Adjoint Ideal]\label{aaadj}
Let $X$ be  a complex manifold and  $Y \subset X$ an irreducible reduced hypersurface. For a quasi-psh function $\vp$ on $X$, we define the analytic adjoint ideal sheaf $\aadj(X,Y;\varphi)$ to be the ideal sheaf of germs of $u \in \OO_{X,x}$ satisfying, for some $\ep > 0$,
\begin{equation} \label{adjn}
    \int_{U \cap X}{|u|^2 e^{-(1+\ep)\vp} \frac{1}{|s|^{2}(\log{|s|})^2} dV_\omega } < +\infty
\end{equation}
where $Y = \dv (s)$ on a neighborhood $U$ of $x$ and $dV_\omega$ is the volume form associated to a hermitian metric $\omega$ on $X$.

\end{definition}

 \begin{remark1}
 The need to take such $\ep$'s in Definition~\ref{aaadj} arises even when $\vp$ has analytic singularities (see the example after Definition 2.7 in \cite{G12}). Also it is known from \cite{GL} that the analytic adjoint ideal sheaf may not be coherent when $\vp$ does not have analytic singularities. 
 
\end{remark1}

 We have the following  generalization of \cite[Prop. 2.11]{G12} whose method of proof is adapted to the current setting. 

\begin{theorem}\label{AAAdj}
 Let $X, Y, B$ be as in Definition~\ref{Adj}. Let $\varphi_B$ be a quasi-psh function on $X$ with poles along $B$. Then we have
    $$
        \aadj(X,Y;\varphi_B) = \adj(X,Y;B).
    $$
\end{theorem}

\begin{proof}

At each point $x \in Y$, we will compare the conditions to belong to the stalk of both sides. For this purpose, we may certainly replace $X$ with a relatively compact neighborhood of $x$. 
Let $f : X' \to X$ be a log resolution of the pair $(X,Y+B)$ with exceptional divisors $E_1,\ldots , E_m$. Let $s$ be a local holomorphic function such that $Y = \dv (s)$ near $x$.  
\\

\noindent \emph{\bf  Step 1}. 
By the definition of $\aadj(X,Y;\varphi_B)$, we have $g \in \aadj(X,Y;\varphi_B)_x$ if and only if there exist $\epsilon >0$ and an open neighborhood $U \subset X$ of $x$ such that
\begin{equation*}
    \int_{U }{ |g|^2 e^{-(1+\epsilon)\varphi_B} \frac{1}{|s|^{2} (\log{|s|})^2} dV_\omega} < +\infty.
\end{equation*}
By change of variables, this is equivalent to the following:
\begin{equation} \label{ChgOfVar}
    \int_{f^{-1}(U)}{ |g \circ f|^2 e^{-(1+\epsilon)\varphi_B \circ f } \frac{1}{|s\circ f|^{2}(\log{|s\circ f|})^2 } |J_f|^2 dV_{\omega'} } < +\infty
\end{equation}
where $|J_f|$ is the Jacobian determinant of $f$ and $dV_{\omega'}$ is a smooth volume form associated to a hermitian metric $\omega'$ on $X'$.

We choose a finite number of local analytic coordinate (relatively compact) neighborhoods $V_1, \ldots, V_N$ covering $f^{-1}(x)$ on each of which $\Supp (K_{X'}- f^* (K_X + Y + B) + \exc(f)) = \Supp (f^* (Y+B) + \exc(f)) $ is given by coordinate hyperplanes.
 In this case, the condition (\ref{ChgOfVar}) holds if and only if the integral (\ref{ChgOfVar}) over $V_j \cap f^{-1}(U )$ instead of $f^{-1}(U)$ is finite for all $j$. For  convenience, we let $V := V_j$ and $U' := f^{-1}(U ) \cap V_j$.
Writing $B = \sum_{i=1}^{k}{b_iB_i}$ on $U$, we denote by $B'_i$ the strict transform of $B_i$ on $U'$ for each $i$. Then we may write
\begin{equation} \label{pullbackEq}
     f^* ( K_X +  Y +  B)|_{U'}= K_{X'}|_{U'} + Y' \cap U' + \sum_{i=1}^{k}{b_i   (B'_i \cap U')} + \sum_{i=1}^{l}{e_i  (E_i \cap U')}
\end{equation}
on $U'$ for some $b_i, e_i \in \QQ$ and some $k, l \in \ZZ_+$ with $k+l \le n$. Note that $Y'$ may not intersect $U'$. 

We have two cases: (i) $Y' \cap U' \neq \emptyset$ and (ii) $Y' \cap U' = \emptyset$. In Step 2 and Step 3 (respectively for (i) and (ii)),  we will identify the equivalent condition to (\ref{ChgOfVar})
 in terms of the exponents of $g \circ f$ assuming $g \circ f$ to be a monomial without loss of generality. The equivalent condition is given as \eqref{exponentscondition} (and equivalently \eqref{MonomialCondition}).  In Step 4, we will complete the proof  by comparing such a condition with the definition of algebraic adjoint ideals.
\\

\noindent \emph{\bf Step 2}.
Let us consider the first case $Y' \cap U' \neq \emptyset$.  Assume that $Y' = (w_n = 0)$, $B'_i = (w_i = 0)$ and $E_j = (w_{k+j} = 0)$ for $1 \le i \le k$ and $1 \le j \le l$ using coordinates $(w_1,\ldots, w_{k+l}, w_{k+l+1}, \ldots, w_n)$ for $U'$. Also we may write
\begin{equation} \label{pullbackFtn}
\begin{split}
    \varphi_B \circ f &= \log{ \left( \prod_{i=1}^{k+l}{ |w_i|^{2b_i} } \right) } + \bigO(1), \\
    |s \circ f| &= |w_n| \prod_{i=k+1}^{k+l}{ |w_i|^{ c_i} } , \\
    \abs{J_f}^2 &= \prod_{i=k+1}^{k+l}{ |w_i|^{2d_i} }
\end{split}
\end{equation}
on $U'$ for some nonnegative $c_i, d_i \in \ZZ$. After rearranging $w_{k+1}, \ldots, w_{k+l}$, we may assume that there exists $k_1$ with $k \le k_1$ such that $b_i = 0$ if and only if $k_1+1 \le i \le k+l$. Using \eqref{pullbackFtn}, the condition (\ref{ChgOfVar}) is equivalent to
\begin{equation} \label{coorExpr}
    \int_{U'}{ \frac{|g\circ f|^2 \prod_{i=1}^{k+l}{|w_i|^{-2(1+\epsilon)b_i} } \prod_{i=k+1}^{k+l}{ |w_i|^{-2c_i + 2d_i} }}{ |w_n|^2 \left\{ \log{\left( |w_n| \prod_{i=k+1}^{k+l}{ |w_i|^{c_i} } \right)} \right\}^2 } dV_{\omega'}} < +\infty.
\end{equation}
Using the Parseval identity and the Taylor expansion of $g \circ f$, we may assume that $g \circ f$ is a monomial $w^\alpha = \prod{w_i^{\alpha_i}}$ for some integers $\al_i$. Thus (\ref{coorExpr}) yields
\begin{equation}
    \int_{U'}{ \frac{ \prod_{i=1}^{k+l}{|w_i|^{2\alpha_i-2(1+\epsilon)b_i} } \prod_{i=k+1}^{k+l}{ |w_i|^{-2c_i + 2d_i} } \prod_{i>k+l}{|w_i|^{2\alpha_i}} }{ |w_n|^2 \left\{ \log{\left( |w_n| \prod_{i=k+1}^{k+l}{ |w_i|^{c_i} } \right)} \right\}^2 } dV_{\omega'}} < +\infty.
\end{equation}
With polar coordinates, the above integral becomes
\begin{equation} \label{inPolar}
    C \int_{[0,\delta]^n}{
    \frac{
    r_n^{-2}
    \prod_{i=1}^{k+l}{ r_i^{2\alpha_i - 2(1+\epsilon)b_i +1 } }
    \prod_{i=k+1}^{k+l}{ r_i^{-2c_i + 2d_i} }
    \prod_{i>k+l}{ r_i^{2\alpha_i + 1} }
    }
    {
    \left\{ \log{\left( r_n \prod_{i=k+1}^{k+l}{ r_i^{c_i} } \right)} \right\}^2
    }
    dr_1 \cdots dr_n
    }
\end{equation}
for sufficiently small $\delta >0$ and for some constant $C > 0$. If we set $\lambda_i(\epsilon)$ by
\begin{equation*}
\lambda_i (\epsilon) = \left\{
\begin{array}{ll}
2\alpha_i -2(1+\epsilon)b_i + 1 & (1\le i \le k), \\
2\alpha_i -2(1+\epsilon)b_i - 2c_i + 2d_i + 1 & (k+1 \le i \le k+l), \\
2\alpha_i + 1 &(k+l+1 \le i \le n-1 ), \\
2\alpha_n - 1 & (i = n), 
\end{array} \right.
\end{equation*}
then the integral (\ref{inPolar}) can be written as
\begin{equation} \label{PolarWithLambda}
    \int_{[0,\delta]^n}{
    \frac{
    \prod_{i=1}^{n}{ r_i^{\lambda_i(\epsilon)} }
    }
    {
    \left\{ \log{\left( r_n \prod_{i=k+1}^{k+l}{ r_i^{c_i} } \right)} \right\}^2
    }
    dr_1 \cdots dr_n
    }.
\end{equation}
When $\epsilon$ is sufficiently small, we claim that the integral \eqref{PolarWithLambda} is finite for arbitrary $\epsilon'$ with $0< \epsilon' <\epsilon$ if and only if we have
\begin{equation} \label{exponentscondition}
\begin{split}
    \lambda_i (\epsilon) & \ge -1 \quad (1\le i \le k_1), \\
    \lambda_i (\epsilon) & > -1 \quad (k_1+1 \le i \le k+l),
\end{split}
\end{equation}
or equivalently, we have
\begin{equation}
\begin{array}{l@{\hfill}ll} \label{MonomialCondition}
    \alpha_i \,\, & \ge  (1+\epsilon)b_i -1 & (1\le i \le k), \\
    \alpha_i \,\, & \ge (1+\epsilon)b_i + c_i - d_i -1 & (k+1 \le i \le k_1), \\
    \alpha_i \,\, & > c_i - d_i -1 & (k_1+1 \le i \le k+l).
\end{array}
\end{equation}

Indeed, the above integral is finite  for every $\epsilon' < \epsilon$ if and only if $\lambda_i(\epsilon) \ge -1$ for all $i$ and $\lambda_i(\epsilon) > -1$ for all $i$ with $b_i = 0$ except for at most one $i_0$. If $\alpha_n = 0$, then the equivalence of (\ref{PolarWithLambda}) and (\ref{exponentscondition}) can be readily checked. Suppose that $\alpha_n > 0$. From the second line of (\ref{pullbackFtn}) and the inclusion $f(E_i) \subset Y$ under the assumption $c_i > 0$, we obtain that for $k_1 + 1 \le i \le k+l$,
$$
\lambda_{i}(\epsilon) = 2\alpha_{i} - 2c_{i} + 2d_{i} + 1 \ge 2 c_{i} (\alpha_n - 1) + 2d_{i} + 1 \ge 0,
$$
which shows that (\ref{exponentscondition}) is satisfied.
\\

\noindent \emph{\bf Step 3}.
 Now consider the second case, $Y' \cap U' = \emptyset$. The only difference from the first case (in Step 2) is that the second equation in (\ref{pullbackFtn}) now becomes
$$
|s \circ f| = \prod_{i=k+1}^{k+l} |w_i|^{c_i}.
$$
and the integral corresponding to (\ref{PolarWithLambda}) is given by 
\begin{equation} \label{PolarWithLambda2}
    \int_{[0,\delta]^n}{
    \frac{
    \prod_{i=1}^{n}{ r_i^{\lambda_i(\epsilon)} }
    }
    {
    \left\{ \log{\left( \prod_{i=k+1}^{k+l}{ r_i^{c_i} } \right)} \right\}^2
    }
    dr_1 \cdots dr_n
    }
\end{equation}
where
\begin{equation*}
\lambda_i (\epsilon) = \left\{
\begin{array}{ll}
2\alpha_i -2(1+\epsilon)b_i + 1 & (1\le i \le k), \\
2\alpha_i -2(1+\epsilon)b_i - 2c_i + 2d_i + 1 & (k+1 \le i \le k+l), \\
2\alpha_i + 1 &(k+l+1 \le i \le n). \\
\end{array} \right.
\end{equation*}

Let $k_1$ be an integer as was determined in Step 2 (after \eqref{pullbackFtn}). Our claim is that the integral (\ref{PolarWithLambda2}) is finite for all sufficiently small $\epsilon > 0$ if and only if  (\ref{exponentscondition}) holds. As checking the necessity is elementary, it is enough to show the sufficiency. Since the integrability of (\ref{PolarWithLambda2}) for $\epsilon'$ with $0 < \epsilon' < \epsilon$ implies $\lambda_{i}(\epsilon) \ge -1$ when $i \le k_1$, it remains to show the following lemma.

\begin{lemma}
If the integral (\ref{PolarWithLambda2}) is finite, then $\lambda_{k+i} (\epsilon) > -1$ for $k_1 + 1 \le k+ i \le n$.
\end{lemma}

\begin{proof}
Suppose that there is an index $i_0$ such that $b_{k+i_0} = 0$ and $\lambda_{k+i_0}(\epsilon) = -1$. This implies that
$$
c_{k+i_0} = \alpha_{k+i_0} + d_{k+i_0} + 1 \ge 1
$$
and the center of $E_{i_0}$ with respect to $f$ is contained in $Y$. Note that such $E_{i_0}$ with $\lambda_{k+i_0} (\epsilon) = -1$ cannot meet $Y'$ by the argument in the second step.

 Since $f$ can be locally taken to be the composition of a finite number of blow-ups along smooth centers \footnote{Here we use the log resolution in the analytic category, cf. \cite{H64}, \cite{AHV77}, \cite{W09}. In particular,  we use \cite[Thm. 2.0.2, 2.0.3]{W09} as $U'$ is taken to be relatively compact from Step 1.}, we may decompose $f$ as the composition of two proper bimeromorphic morphisms   $f_2 : X' \to X_1$ and $f_1 : X_1 \to X $ such that an exceptional divisor $E$ on $X'$ is not exceptional with respect to $f_2$ if and only if $f_2(E_{i_0})$ is contained in $f_2(E)$. Then there exist exceptional divisors of $f$, say $E_1, \ldots, E_{i_0-1}$ (after rearranging the order of coordinates), so that the set $\{ E_1,\ldots, E_{i_0-1}\}$ is equal to the set of all exceptional divisors $E$ of $f$ satisfying
\begin{enumerate}
\item $E$ is not exceptional with respect to $f_2$, and
\item $f_2(E_{i_0})$ is contained in $f_2(E)$. 
\end{enumerate}
 
  Here, after replacing the index $i_0$ by a smaller index if necessary, we may assume that either
\begin{enumerate}
\item $\lambda_{k+i} (\epsilon) > -1$ and $b_{k+i} = 0$, or
\item $\lambda_{k+i} (\epsilon) = -1$ and $b_{k+i} > 0$
\end{enumerate}
holds for each $i=1,\ldots,i_0-1$. We note that $\lambda_{k+i}(\epsilon) + 2(1+\epsilon)b_{k+i} \ge 1$ in both cases. This follows from the observation that $\lambda_{k+i}(\epsilon) + 2(1+\epsilon)b_{k+i}$ is an odd integer greater than $-1$ in the former case and $b_{k+i}$ is a positive integer in the latter case.

On $X_1$, assume that (the corresponding divisor to) $E_i$ is given by $\mathrm{div}(s_i)$ locally for each $i = 1, \ldots, i_0-1$ and write
$$
s_i \circ f_2  = w_{k+i} \cdot w_{k+i_0}^{\beta_i} \cdot u_i
$$
where $\beta_i > 0$ and $u_i$ is a holomorphic function that is not generically vanishing on $E_1,\ldots, E_{i_0-1}$ and $E_{i_0}$. Since the codimension of $f_2(E_{i_0})$ is at least $2$, we obtain the following:
\begin{equation*}
\begin{split}
\alpha_{k+i_0} &\ge \sum_{i=1}^{i_0-1} \beta_i \alpha_{k+i} , \\
d_{k+i_0} & \ge \sum_{i=1}^{i_0-1} \beta_i d_{k+i} , \\
c_{k+i_0} & = \sum_{i=1}^{i_0-1} \beta_i c_{k+i}.
\end{split}
\end{equation*}
This implies that
\begin{align*}
\lambda_{k+i_0} (\epsilon) &\ge 1 + \sum_{i=1}^{i_0-1}{ 2\beta_i (\alpha_{k+i} - c_{k+i} + d_{k+i})} \\
& = 1 + \sum_{i=1}^{i_0-1}{ \beta_i \{ \lambda_{k+i}(\epsilon) + 2(1+\epsilon)b_{k+i} - 1 \} } \\
& \ge 1,
\end{align*}
which is a contradiction. This completes the proof of the lemma.
\end{proof}
\;
\\

\noindent \emph{\bf Step 4}.
Finally we consider the algebraic adjoint ideal.
Combining (\ref{pullbackEq}) and (\ref{pullbackFtn}), we have
\begin{equation*}
e_i = b_{k+i} + c_{k+i}  - d_{k+i}.
\end{equation*}
By the definition of $\adj(X,Y;B)$, we see that $g \in \adj(X,Y;B)_x$ if and only if
\begin{equation}
    g \circ f \in \OO_{X'}( -\sum{ \lfloor b_i \rfloor B'_i } - \sum{ \lfloor b_{k+i} + c_{k+i} - d_{k+i} \rfloor E_i } ) (f^{-1}(U))
\end{equation} 

\noi where $ \lfloor x \rfloor$ is the greatest integer less than or equal to $x$. 
 Therefore we have $g \in \adj(X,Y;B)$ if and only if
\begin{equation*}
\begin{array}{l@{\hfill}ll}
    \ord_{B'_i}(g\circ f) \,\, & \ge \lfloor b_i \rfloor & (1 \le i \le k), \\
    \ord_{E_i}(g \circ f) \,\, & \ge \lfloor b_i + c_i - d_i \rfloor & (k+1 \le i \le k + l).
\end{array}
\end{equation*}
which is equivalent to (\ref{MonomialCondition}). This completes the proof of Theorem~\ref{AAAdj}. 
\end{proof}

\medskip

\section{Proofs of the main results}

 In this final section, based on previous preparations, we complete the proofs of the results (\ref{dvpsi3}), (\ref{dvpsi}), (\ref{main}), (\ref{neg3}) along with  presenting some examples related to these results.

\;

\begin{proof}[\textbf{Proof of Theorem~\ref{dvpsi3}}]

Let $f: X' \to X$ be a log resolution of the pair $(X, Y+ B)$ and write $$K_{X'} + \Delta_{X'}  = f^* (K_X + \Delta) = f^* (K_X + Y+ B).$$ Let $Y'$ be the strict transform of $Y$ on $X'$ and $\tilde{f} : Y' \to Y^\nu$ the induced morphism to the normalization $Y^\nu$ of $Y$.  Then define $\Delta'$ by 
$K_{X'} + \Delta' = f^* (K_X + Y) .$ Note that we have $\Delta_{X'} - Y' = (\Delta' - Y') + f^*B$.

We will use Proposition~\ref{iber2} (3') (taking $D=0$ there). Let $\alpha$ denote the measure $d\nu$ in \eqref{dvp1} so that the direct image $f_* \alpha$ is equal to $dV[\psi]$. 
 By the projection formula of the direct image~\cite[I (2.14) (c)]{DX}, the direct image along $f$ of the measure $(f^* e^{-\vp_B}) \alpha$  is equal to $e^{-\vp_B} f_* \alpha = e^{-\vp_B} dV[\psi]$. We note that $(f^* e^{-\vp_B}) \alpha$ on $Y'$ has poles along the divisor $$( (\Delta' - Y') + f^*B)|_{Y'} = (\Delta_{X'} - Y')|_{Y'}.$$ Since this is an snc divisor, we see that  the measure $\beta:= (f^* e^{-\vp_B}) \alpha$ is locally integrable if and only if the pair $(Y', (\Delta_{X'} - Y')|_{Y'})$ is klt.  Hence this is equivalent to 
the measure $e^{-\vp_B} dV[\psi]$ being locally integrable since this is the direct image (i.e. change of variables in this case) of $\beta$. 

  On the other hand, we have $K_{Y'} + (\Delta_{X'} - Y')|_{Y'} \sim_{\QQ}  \tilde{f}^*  (K_{Y^{\nu}} + \Diff (B) )$ by Proposition~\ref{prop1}.  Thus we have the equivalence :  the pair $(Y^{\nu},  \Diff (B))$ is klt if and only if  the pair $(Y', (\Delta_{X'} - Y')|_{Y'})$ is klt from  \cite[(3.10.2)]{Ko97}. Combining these, the proof is complete. 
 \end{proof}

\begin{proof}[\textbf{Proof of Corollary~\ref{dvpsi}}]

  Applying Theorem~\ref{dvpsi3} to $B=0$ and $\vp_B = 0$, we first have that $dV[\psi]$ is locally integrable if and only if the pair $(Y^{\nu}, \Diff 0)$ is klt.  From this, one direction (`if') is clear using Proposition~\ref{diff0}.

  To prove the other direction,  assume that $dV[\psi]$ is locally integrable.  First we will check that $Y$ is normal. 
  Let $\widetilde{\OO}_{Y}$ be the sheaf of weakly holomorphic functions on $Y$, i.e. holomorphic functions $f$ on $Y_{\reg}$ such that every point of $Y_{\sing}$ has a neighborhood $V$ with $f$ being bounded on $Y_{\reg} \cap V$. It suffices to show that 
 $\widetilde{\OO}_{Y} = \OO_{Y}$ by \cite[Chap.II \S 7]{DX}. 

 Let $f$ be a weakly holomorphic function germ $f$ on $Y$. 
 Since the constant function $1$ is locally integrable with respect to $dV[\psi]$ and $f$ is locally bounded, it follows that $f$
  is locally $L^2$ with respect to $dV[\psi]$ at every point $q \in Y$. Hence we can apply Theorem~\ref{extn} to a relatively compact Stein neighborhood $U \subset X$  of $q$ and extend $f$ from $U \cap Y_{\reg}$ to a holomorphic function $F$ on $U$. The restriction of $F$ back to $Y \cap U$ is holomorphic, hence  we get $\widetilde{\OO}_{Y} = \OO_{Y}$. 
  
  Now that $Y$ is normal, by Proposition~\ref{diff0}, we have $(Y^{\nu}, \Diff 0) = (Y, 0)$. As is well known, cf. \cite[Thm. 11.1]{Ko97}, $(Y,0)$ is klt if and only if $(Y,0)$ is canonical since $Y$ is a Cartier divisor in a complex manifold. 
\end{proof}

 The following is an example where $\Diff (0) = 0$ but $dV[\psi]$ is not locally integrable.

 \begin{example}
 
Suppose $X$ is smooth of dimension $n \ge 3$. Let $Y \subset X$ be an irreducible normal hypersurface which is smooth except at a point $p \in Y$.  Then we have $\Diff (0) = 0$ by Proposition~\ref{diff0}. 
Moreover, suppose that the multiplicity of $Y$ at $p$ is sufficiently high so that $(Y, 0)$ fails to be klt. By Corollary~\ref{dvpsi}, $dV[\psi]$ is not locally integrable at $p$. 
 
 \end{example} 

We also have the following related example. 

\begin{example}\label{Fermat}

 Let $Y \subset X = \CC^n$ be given by $f(z_1, \ldots, z_n) = z_1^d + \ldots + z_n^d = 0$ where $n \ge 3$ and $d \ge 1$.  Then $Y$ is normal and by Proposition~\ref{diff0}, $\Diff 0 = 0$.   Take $\psi = \log \abs{f}^2$.  By Corollary~\ref{dvpsi}, the Ohsawa measure $dV[\psi]$ is locally integrable if and only if   $(Y,0)$ is a canonical pair if and only if $d \le n-1$. 

 On the other hand, let $\wt{dV_Y} := \frac{1}{\abs{df}^2} dV$ where $dV$ is a smooth volume form on $Y$ (i.e. induced from a smooth form on $X$). Since $\frac{1}{\abs{df}^2} = \frac{1}{\sum \abs{z_j}^{d-1}}$ (up to a positive locally bounded factor), $\wt{dV_Y}$ is locally integrable if and only if $d - 1 \le n-2$ by \cite[Lem. 2.3]{GL18}.  In particular, for a holomorphic function, having finite $L^2$ norm with respect to $\wt{dV_Y}$ do not force it to vanish along the singular locus of $Y$ if $d \le n-1$. 

\end{example}

Now we turn to the proof of Theorem~\ref{main} regarding the inversion of adjunction. 

\begin{proof}[\textbf{Proof of Theorem~\ref{main}}]

 One direction (if)  is easy and well-known from algebraic geometry and we refer the readers to \cite[Thm. 7.5]{Ko97}.
 
   To prove the other direction using $L^2$ extension, suppose that $(Y^{\nu}, \Diff(B) )$ is klt. We need to show that the pair $(X, Y+B)$ is plt in a neighborhood of $Y$.  The problem is local since the notions of klt and plt are local in the analytic topology as is well known, cf. \cite[(3.7.6)]{F07}. Therefore we may restrict to a Stein open subset $U \subset X$ whose closure is compact.  We may also assume that the line bundles $K_X, L, M$ in Theorem~\ref{extn}  are trivialized on $U$. For simplicity of notation, we  denote $U$ by $X$ again. 
   
   In Theorem~\ref{extn}, we will take $e^{-\vp}$ to be $e^{-(1+\ep)\vp_B}$ for some $\ep >0$ where $\vp_B$ is a  psh function with poles along the effective divisor $B$.  Since $X$ is smooth, we have $\Diff ((1+\ep)B) = \Diff(0) + (1+\ep) B|_{Y^{\nu}}$. Hence the pair $(Y^{\nu}, \Diff( (1+\ep)B))$ is still klt for $\ep > 0$ sufficiently small. 
   It follows that, by Theorem~\ref{dvpsi3}, the constant function $1$ on the regular locus of $Y$ is $L^2$ with respect to $e^{- (1+\ep)\vp_B} dV[\psi]$.  Now we  apply the $L^2$ extension Theorem~\ref{extn}  to extend $u=1$ to a holomorphic function $\ut$ satisfying \eqref{eq1}. (Note that the curvature conditions are easily satisfied since $X$ is Stein.) 
 
  From the $L^2$ estimate \eqref{eq1}, we see that $\ut$ belongs to the analytic adjoint ideal sheaf $\AAdj (X, Y; \vp_B)$ (see Definition~\ref{aaadj}) which is equal to the algebraic adjoint ideal sheaf $\Adj(X, Y; B)$ by Theorem~\ref{AAAdj}. Since the restriction of $\ut$ to $Y$ is equal to $1$, $\Adj(X, Y; B)$ is trivial in a neighborhood of $Y$. Thus the pair $(X, Y+B)$ is plt in a neighborhood of $Y$, which concludes the proof. 
\end{proof}

\begin{remark1}\label{kawadiff}

 In the above proof of Theorem~\ref{main}, in fact, we could also use the $L^2$ extension theorem of \cite{K10} instead of Theorem~\ref{extn} (cf. Theorem~\ref{dvpsi3}).  However the $L^2$ extension of \cite[Main Thm.]{K10} is formulated for $Y \subset X$ of general codimension in the setting of log canonical centers and Kawamata metrics (where $X$ is a normal variety). Instead of taking what we need from that statement using the relation between Kawamata metrics and the `differents' (cf. \cite[p.140, (4)]{H14}, \cite{Ko07}), we used Theorem~\ref{extn} of \cite{D15} along with recalling the theory of `differents' in \S 2.

\end{remark1}

Finally, we give the proof of Theorem~\ref{neg3} which generalizes (in codimension $1$) a negative result of Guan and Li~\cite{GL18}. 

\begin{proof}[\textbf{Proof of Theorem~\ref{neg3}}]

 Suppose that such a version of the  $L^2$ extension Theorem~\ref{extn0} holds, i.e. having a locally integrable measure $dV_Y$ in the place of $dV[\psi]$. Then we can apply the theorem to $u=1$ and $\vp = 0$ to obtain an extension $\ut$ satisfying  that the LHS of the estimate \eqref{est1} is finite, i.e.  $$ \int_{X} \abs{\ut}^2  \gamma (\delta \psi) e^{-\psi} dV_X < +\infty $$ where $\psi$ is quasi-psh with poles along $Y$.  Hence by the same arguments using the definition of the analytic adjoint ideal \eqref{adjn} and Theorem~\ref{AAAdj} in the last paragraph of the proof of Theorem~\ref{main}, the pair $(X, Y)$ is plt in a neighborhood of $Y$. It follows from Proposition~\ref{prop1} that $(Y^\nu,  \Diff 0)$ is klt, which is contradiction. 
\end{proof}

\footnotesize

\bibliographystyle{amsplain}

\quad
\\

\normalsize

\noi \textsc{Dano Kim}

\noi Department of Mathematical Sciences and Research Institute of Mathematics

\noi Seoul National University, Seoul 08826, Korea

\noi e-mail: kimdano@snu.ac.kr

\qa
\\

\noi \textsc{Hoseob Seo}

\noi Research Institute of Mathematics

\noi Seoul National University, Seoul 08826, Korea
\\

\noi Current address of Hoseob Seo : 

\noi Center for Complex Geometry, Institute for Basic Science (IBS)

\noi 55 Expo-ro, Yuseong-gu, Daejeon 34126, Korea

\noi e-mail: hskoot@snu.ac.kr; hskoot@ibs.re.kr

\end{document}